\documentclass{article}
\usepackage{amsmath, amssymb, amsthm}
\usepackage{algorithm, algpseudocode}
\usepackage{hyperref, cleveref}
\usepackage{multirow}
\usepackage{tikz, graphics, caption}
\usepackage[left=20mm, right=20mm, top=20mm, bottom=20mm]{geometry}
\usetikzlibrary{decorations.pathreplacing,decorations.markings,calligraphy}

\title{Presentations for Singly-Cusped Bianchi Groups}
\author{Tanner Reese}

\newcommand\ab[1]{#1^{\mathrm{ab}}}

\newcommand\conj[1]{\overline{#1}}

\newtheorem{macthm}{Macbeath's Theorem}
\newtheorem{lem}{Lemma}
\newtheorem{defn}{Definition}
\newtheorem{thm}{Theorem}

\begin{document}
\maketitle

\begin{abstract}
We produce a complete list of group presentations for singly-cusped Bianchi groups,
$ PSL_2 (\mathcal{O}_d) $ where $ \mathcal{O}_d $ is the ring of integers for $ \mathbb{Q}(\sqrt{d}) $
and $ d $ is $ -1 $, $ -2 $, $ -3 $, $ -7 $, $ -11 $, $ -19 $, $ -43 $, $ -67 $, or $ -163 $.
To do this, we apply a theorem due to Macbeath to a sufficiently large horoball,
treating the Bianchi groups as discrete subgroups of isometries for $ \textbf{H}^3 $.
As far as we know, explicit presentations were not previously known when $ d $ is $ -43 $, $ -67 $, or $ -163 $.
\end{abstract}

\section{Introduction}

Hyperbolic spaces admit rather intricate groups of isometries.
Studying the geometry of these discrete subgroups
can give useful information about the algebraic structures underpinning these groups.
As a particular case, we can consider discrete subgroups of orientation-preserving isometries
of three-dimensional hyperbolic space, $ \textbf{H}^3 $, known as Kleinian groups.
The group of orientation-preserving isometries of $ \textbf{H}^3 $
can be identified with $ PSL_2 (\mathbb{C}) $ using M\"obius transformations.
The three-dimensional case is of special interest,
because of the role of hyperbolic 3-manifolds in studying 3-manifold topology.

When defining discrete subgroups of isometries, there are broadly two approaches.
One may consider the subgroup geometrically,
in which we view it as being generated by isometries of certain geometric types, such as reflections.
In the particular case of subgroups generated by reflections,
H.S.M. Coxeter classified these groups for spherical and Euclidean spaces of arbitrary dimension, see \cite{Cox}.
Though there are many important open questions about hyperbolic reflection groups,
the construction of these groups provides an immediate description of their group presentations and fundamental domains, see \cite{AVS}.

The other approach is to consider the group of isometries as a matrix group, in this case $ PSL_2 (\mathbb{C}) $.
Then, we can naturally construct subgroups by considering matrices
whose coefficients are algebraic integers of a particular number field.
Of special interest are those subgroups which have finite covolume and so form lattices.
Unlike geometrically constructed subgroups, these groups do not have easily determined group presentations.
The simplest example of these are those classical matrix groups with integer coefficients such as $ PSL_2(\mathbb{Z}) $.
The special linear group on $ \mathbb{Z} $ has been studied and presentations determined for all dimensions.
Steinberg determined presentations for $ SL_n(\mathbb{Z}) $ with $ n \geq 3 $ (see \cite{St}),
while $ SL_2(\mathbb{Z}) $ is classical.

The Bianchi groups arise by restricting the components to the algebraic integers over a complex quadratic number field.
Every complex quadratic number field can be represented as an extension $ \mathbb{Q}(\sqrt{d}) $
for some negative square free integer $ d $.
Then, its set of algebraic integers is the ring of integers $ \mathcal{O}_{\mathbb{Q}(\sqrt{d})} = \mathcal{O}_d $.
It is well known that $ \mathcal{O}_d = \mathbb{Z}[\omega] $ where
\[
\omega = \begin{cases}
\sqrt{d} & \text{if } d \equiv 2, 3 \pmod{4} \\
\frac{1 + \sqrt{d}}{2} & \text{if } d \equiv 1 \pmod{4}
\end{cases}.
\]
The Bianchi group $ \Gamma_d $ is defined as
\[ \Gamma_d = PSL_2 (\mathcal{O}_d) \]
When clear from context, we may omit $ d $ and use
$ \Gamma $ to refer to a Bianchi group and $ \mathcal{O} $ to refer to its ring of integers.
These are, in some sense, the simplest examples of isometry subgroups constructed in this way
besides those with integer coefficients.
Because these rings of integers are complex and quadratic, they can be viewed as lattices in the complex plane.
This simplifies their characterization in comparison to real number fields or those of higher degree.

The goal of this paper is to find presentations for singly cusped $ \Gamma_d $.
Because the number of cusps of $ \textbf{H}^3 / \Gamma_d $ is equal to the class number of $ \mathcal{O}_d $ (see \cite{MR}),
we need only consider those $ d $ for which $ \mathcal{O}_d $ is a principal ideal domain.
It was conjectured by Gauss and proven by Heegner (with modification by Stark)
that the only such $ d $ values are $ \{-1, -2, -3, -7, -11, -19, -43, -67, -163\} $, see \cite{Sta}.
To actually find a group presentation, we will consider the orbit of a horoball $ V $ under $ \Gamma_d $.
Assuming its orbit covers the space, we can use Macbeath's Theorem (see Theorem \ref{macthm}) to determine
a set of generators and relations from the horoballs and their intersections.
Since $ \textbf{H}^3 / \Gamma_d $ has only one cusp,
it is ensured that a single sufficiently large horoball will exist whose orbit covers the space.

The approach used here is adapted from the one presented by Mark and Paupert, see \cite{MP}.
It differs from the typical method of determining face-pairings of the fundamental domain.
Yasaki considers Vorono\"i polyhedra and their stabilizer subgroups to find presentations, see \cite{Yas}.
Page uses a method in which face-pairings are generated until a complete set is found, see \cite{Page}.
Unlike these methods, here a fundamental domain for the group is not determined directly
nor are a limited set of relations determined.
Instead, a large number of generators and relations are found
before using purely algebraic methods to simplify the presentation.

Ultimately, we have found presentations for the cases of $ d = -43, -67, -163 $
which along with the previously known cases of $ d = -1, -2, -3, -7, -11, -19 $
(discovered by Swan, see \cite{Sw}) completely describes the singly cusped Bianchi groups.
The abelianizations of the novel presentations were found to be
\[ \begin{split}
\ab{\Gamma_{-43}} &= C_\infty^2 \\
\ab{\Gamma_{-67}} &= C_\infty^3 \\
\ab{\Gamma_{-163}} &= C_\infty^7
\end{split} \]
These group presentations can be used to compute the homology of the quotient space.
Specifically, $ \ab{\Gamma_d} = H^1(\textbf{H}^3 / \Gamma_d, \mathbb{Z}) $.
Further, a better understanding of the group structure provides information about
a variety of topological information as demonstrated by \c{S}eng\"un, see \cite{Sen}.

The paper is structured as follows.
In section two, we review properties of hyperbolic three-space
and introduce the model of $ \textbf{H}^3 $ that will be used.
In section three, we present our method using Macbeath's Theorem
including implementation details for the algorithms.
In section four, we present the full group presentations.

\medskip
I would like to thank my advisor, Julien Paupert.
He introduced me to this method for determining group presentations,
advised me on its implementation,
and edited the various drafts of this paper.
I would also like to thank Nancy Childress for reviewing and editing.

\section{Hyperbolic Space} \label{hyp_space}

Hyperbolic spaces, $ \textbf{H}^n $, are simply connected Riemannian manifolds of constant negative curvature.
For a given dimension, such a space is unique up to scaling of the curvature.
As a general reference for information about hyperbolic manifolds, one can consult \emph{Foundations of Hyperbolic Manifolds}, see \cite{Rat}.
The work by Swan also provides an introduction to the topology of hyperbolic manifolds, see \cite{Sw}.
To study Bianchi groups, we are interested in the particular case of $ n = 3 $.
While several models exist for representing $ \textbf{H}^3 $ in a convenient form,
we will be using the Poincar\'e half-space model.
In this model,
\[
\textbf{H}^3 = \{(z, \lambda) \in \mathbb{C} \times \mathbb{R} \;:\; \lambda > 0 \}
\]
the boundary of which is represented by the $ \lambda = 0 $ plane along with the point at infinity,
$ \partial\textbf{H}^3 = (\mathbb{C} \times \{0\}) \cup \{\infty\} $.
Here, $ \infty $ represents the value of the limit $ \lim_{\lambda \to \infty} (z, \lambda) $ for any $ z \in \mathbb{C} $.
The Riemannian metric is defined as
\[ ds^2 = \frac{|dz|^2 + d\lambda^2}{\lambda^2} \]
So the distance between two points $ (z_1, \lambda_1), (z_2, \lambda_2) \in \textbf{H}^3 $ is
\[ d((z_1, \lambda_1), (z_2, \lambda_2)) = \mathrm{arccosh}\left(1 + \frac{|z_1 - z_2|^2 + (\lambda_1 - \lambda_2)^2}{2 \lambda_1 \lambda_2}\right) \]
The metric can also be used to define geodesics.
For this model, the geodesics are the semicircles whose center lies in the boundary plane
along with the vertical rays perpendicular to the boundary plane.
Similarly, the planes of the space are semispheres whose center lies in the boundary plane
and vertical planes perpendicular to the boundary plane.

As in Euclidean space, we can also define balls,
\[ B(c, r) = \{x \in \textbf{H}^3 \mid d(c, x) < r \} \]
which clearly will also be closed under isometries.
Further, in the half-space model, these hyperbolic balls will appear as Euclidean balls,
though not with the same centers or radii.
If one takes the limit of a ball as the radius approaches infinity,
while keeping a point with a tangent plane fixed,
then one obtains a \emph{horoball} based at some infinite point.
In the half-space model, every horoball is of one of the two following types.
In the first case, it is a ball tangent to and above the boundary plane,
\[
B = \{ (z, \lambda) \mid |z - z_0|^2 + (\lambda - \lambda_0)^2 < \lambda_0^2 \}
\]
and we say $ B $ is \emph{based} at $ z_0 \in \partial\textbf{H}^3 $.
In the second case, it is a half-space above and parallel to the boundary plane,
\[ B = \{ (z, \lambda) \mid \lambda > h \} \]
for some $ h > 0 $.
We say $ B $ is \emph{based} at $ \infty \in \partial\textbf{H}^3 $ with a \emph{height} of $ h $.
Note that with this definition of height, horoballs of smaller height are in fact larger.
Horoballs, like the geodesics, are a closed set under isometries.

With the metric, we can also consider \emph{isometries}.
That is mappings $ \sigma : \textbf{H}^3 \to \textbf{H}^3 $ which preserve the metric,
\[ d(\sigma(x), \sigma(y)) = d(x, y) \]
for all $ x, y \in \textbf{H}^3 $.
Because the metric is fixed, the sets of geodesics, planes, balls, and horoballs are all closed under isometries.
If an isometry does not cause a reflection of the space, then it is said to be \emph{orientation-preserving}.

It happens that the group of orientation-preserving isometries has a convenient representation in terms of matrices.
We can see this by considering the action of such an isometry on the infinite boundary points, $ \partial \textbf{H}^3 $, of $ \textbf{H}^3 $.
Looking at the Poincar\'e ball model, these infinite boundary points can be viewed as the Riemann sphere, $ \widehat{\mathbb{C}} = \mathbb{C} \cup \{\infty\} $.
It happens that the action of any of these isometries on $ \widehat{\mathbb{C}} $
takes the form of M\"obius transformations.
Further, the action of an isometry on the points at infinity fully characterizes it on the hyperbolic space.
Thus, for any such isometry $ \sigma $, it will send $ z \in \widehat{\mathbb{C}} $ to
\[
\sigma(z) = \frac{az + b}{cz + d}
\]
and $ \sigma(\infty) = \frac{a}{c} $.
Because of the invariance under scaling the numerator and denominator by a common factor,
the group of these isometries is $ PSL_2 (\mathbb{C}) $.
Using the behavior of the geodesics to extend $ \sigma $ from $ \partial\textbf{H}^3 $ to $ \textbf{H}^3 $,
\[
\sigma(z, \lambda) = \left(
\frac{(d - \conj{c}\conj{z})(az - b) - \lambda^2 \conj{c}a}{|cz - d|^2 + \lambda^2 |c|^2},
\frac{|ad - bc| \lambda}{|cz - d|^2 + \lambda^2 |c|^2}
\right)
\]
for any $ (z, \lambda) \in \textbf{H}^3 $, see \cite{Sw}.
\medskip

Throughout this text, we will use $ \Gamma_\infty $ to refer to
the stabilizer subgroup of $ \infty \in \partial \textbf{H}^3 $ in $ \Gamma $.
We will also be using $ T_s \in \Gamma_\infty $ to denote
\[ \begin{bmatrix} 1 & s \\ 0 & 1 \end{bmatrix} \]
which, for any $ s \in \mathcal{O}_d $, is a translation of $ \textbf{H}^3 $.

\section{The Method}

\subsection{Background}

As mentioned above, Macbeath's Theorem allows us to determine a presentation for any subgroup of isometries
using an arbitrary subset, $ V $, of the space.

\begin{macthm} \label{macthm} \cite{Mac}
Let $ X $ be a topological space and let $ G $ be a group acting continuously on $ X $.
Let $ V $ be an open subset of $ X $ whose orbit under $ G $ covers $ X $.

If $ X $ is connected then the elements, $ S = \{g \in G \;:\; g(V) \cap V \neq \emptyset \} $, generate $ G $.
If $ X $ is simply connected then the relations,
\[
R = \{ a^{-1}b = c \;:\; a, b, c \in S, a^{-1}b = c, a(V) \cap b(V) \cap V \neq \emptyset \}
\]
along with the generators $ S $ form a presentation $ G = \langle S | R \rangle $.
\end{macthm}

We will apply this theorem to a horoball based at $ \infty $ of height $ h > 0 $ which we will call $ V $.
Let us consider the orbit of such a horoball.
For each $ \sigma = \begin{bmatrix} a & b \\ c & d \end{bmatrix} \in \Gamma_d $, $ \sigma(\infty) = \frac{a}{c} $.
Thus, $ \sigma(V) $ will be a horoball based at $ \frac{a}{c} \in \mathbb{Q}(\sqrt{-d}) $ of some size.
The orbit of $ V $ consists of horoballs based at rational points along with $ V $.

Since our method is limited to singly cusped quotient spaces, we require $ \mathcal{O}_d $ to be a PID
and so will only be considering Bianchi groups, $ \Gamma_d $ where $ d \in \{ -1, -2, -3, -7, -11, -19, -43, -67, -163 \} $.
Further, because of implementation difficulties, we will only consider cases where $ \mathcal{O}_d $ has no non-trivial units.
There exist non-trivial units $ \sqrt{-1} \in \mathcal{O}_{-1} $ and $ \frac{1 + \sqrt{3}}{2} \in \mathcal{O}_{-3} $ eliminating these cases.

\begin{lem} \label{diamLm}
Let $ V $ be a horoball based at $ \infty $ of height $ h > 0 $.
Suppose $ \sigma \in PSL_2(\mathbb{C}) $ with
\[ \sigma = \begin{bmatrix} a & b \\ c & d \end{bmatrix} \]
does not fix $ \infty $ (that is $ c \neq 0 $).
Then, $ \sigma(V) $ considered as a Euclidean ball in the upper half-space
will have diameter $ \frac{1}{h |c|^2} $.
\end{lem}
\begin{proof}
Let $ B = \mathbb{C} \times h $ be the boundary of $ V $ so
$ \sigma(B) $ is the horosphere bounding $ \sigma(V) $.
Let $ A = (\frac{a}{c}, r) $ be the apex of $ \sigma(B) $
so that $ r $ is the supremum of $ \lambda $ values on the ball.
Then, we consider the vertical geodesic, $ \nu $, running through the apex.
Since the geodesic's endpoints are $ a / c $ and $ \infty $ the geodesic's pre-image, $ \sigma^{-1}(\nu) $
will have endpoints $ \sigma^{-1}(a / c) = \infty $ and $ \sigma^{-1}(\infty) = -d / c $.

We know $ A = \nu \cap \sigma(B) $ so $ \sigma^{-1}(A) = \sigma^{-1}(\nu) \cap B = (-d / c, h) $.
Now, construct the geodesic $ \mu $ tangent to $ B $ at $ \sigma^{-1}(A) $.
In the Euclidean half-space, $ \mu $ will then be a semicircle with radius $ h $.
So we can take one of the endpoints of $ \mu $ to be $ -d / c + h $.
Thus, $ \sigma(\mu) $ will have endpoint
\[
\frac{a\left( h - \frac{d}{c} \right) + b}{c \left( h - \frac{d}{c} \right) + d}
= \frac{hac - ad + bc}{c^2 h - cd + dc}
= \frac{hac - 1}{c^2 h} = \frac{a}{c} - \frac{1}{c^2 h} .
\]

Since $ \mu $ is tangent to $ B $ as $ \sigma^{-1}(A) $,
we know $ \sigma(\mu) $ is tangent to $ \sigma(B) $ at $ A $.
Thus, the center of the semicircle $ \sigma(\mu) $ is $ a / c $.
The distance from $ A $ to $ a / c $ (which is the diameter of $ \sigma(B) $ and $ \sigma(V) $)
is then equal to $ |\sigma(a / c) - \sigma(h - d/c)| = \frac{1}{h |c|^2} $.
In particular, the diameter is less than $ \frac{1}{h} $ for all $ c \neq 0 $.
\end{proof}
\medskip

\begin{lem} \label{ballcoverLm}
Let $ V $ be a horoball based at $ \infty $ of height $ h > 0 $.
Say that $ D $ is a fundamental domain of $ \mathbb{C} / \mathcal{O}_d $ so that
for all $ z \in \mathbb{C} $, there exist $ s \in \mathcal{O}_d $ and $ r \in D $ with $ z = s + r $.
If for all $ r \in D $, $ (r, h) \in \sigma(V) $ for some $ \sigma \in \Gamma $
then $ \Gamma(V) $ covers $ \textbf{H}^3 $.
\end{lem}
\begin{proof}
We prove by contradiction.
Suppose $ S = \Gamma(V) $ does not cover $ \textbf{H}^3 $.
Since $ S $ is open and $ \textbf{H}^3 $ is connected,
$ \textbf{H}^3 \setminus S $ cannot be open so $ S \subsetneq \overline{S} $.
Thus, there exists $ x \in \partial S $.
By local finiteness, there exists a neighborhood $ U $ of $ x $
in which $ \overline{S} \cap U = (\sigma_1 (\overline{V}) \cup \ldots \cup \sigma_n (\overline{V}) ) \cap U $.
Thus, there exists $ \sigma $ such that $ x \in \sigma(\partial V) $.
Taking the pre-image, we get $ \sigma^{-1}(x) \in \partial V = \mathbb{C} \times \{h\} $
so $ \sigma^{-1}(x) = (z + s, h) $ where $ z \in D $ and $ s \in \mathcal{O}_d $.
Contradicting the fact that $ \sigma^{-1}(x) \notin \Gamma (V) $
since by assumption $ x \notin \Gamma(V) $.
\end{proof}
\medskip

\begin{lem} \label{matDefLm}
For all $ a, c \in \mathcal{O} $ with $ \gcd(a, c) = 1 $, there exists a matrix
\[ \sigma = \begin{bmatrix} a & b \\ c & d \end{bmatrix} \]
where $ b, d \in \mathcal{O} $ and $ ad - bc = \det(\sigma) = 1 $.
\end{lem}
\begin{proof}
Since $ \mathcal{O} $ is a PID, $ \langle a \rangle + \langle c \rangle = \langle g \rangle $ for $ g \in \mathcal{O} $.
But then $ g | a $ and $ g | c $ so $ g | \gcd(a, c) = 1 $ so $ g = 1 $.
Thus, $ 1 \in \langle a \rangle + \langle c \rangle $ meaning
there exist $ ad \in \langle a \rangle $ and $ -bc \in \langle c \rangle $ such that $ 1 = ad - bc $.
\end{proof}

\begin{defn}
Let $ a, b \in \mathcal{O} $ with $ \gcd(a, c) = 1 $.
We define the notation $ \mathrm{Mat}\left(\frac{a}{c}\right) $ to mean a matrix
\[ \begin{bmatrix} a & b \\ c & d \end{bmatrix} \in SL_2(\mathcal{O}) \]
for some $ c, d \in \mathcal{O} $.
Such a matrix exists by Lemma \ref{matDefLm}.
\end{defn}

While such a matrix will always exist, it won't be unique.
Pre-multiplication of $ \mathrm{Mat}\left( \frac{a}{c} \right) $ by an element in $ \Gamma_\infty $ will leave $ a $ and $ c $ unaffected.
So we must choose a particular matrix for each $ \mathrm{Mat}\left( \frac{a}{c} \right) $.
\medskip

Using the size of the horoballs in $ \Gamma (V) $ again,
we consider the intersection $ \sigma(V) \cap \partial V $.

\begin{defn}
Let $ V $ be a horoball based at $ \infty $ of height $ h > 0 $.
For any element $ \sigma \in \Gamma \setminus \Gamma_\infty $, we define
\[
\mathrm{Circ}(\sigma) = \{z \in \mathbb{C} \;:\; (z, h) \in \sigma(V) \} = \sigma(V) \cap \partial V .
\]
Further, for any $ a, c \in \mathcal{O}_d $, we say
\[ \mathrm{Circ}\left( \frac{a}{c} \right) = \mathrm{Circ}\left(\mathrm{Mat}\left( \frac{a}{c} \right)\right). \]
\end{defn}
\begin{proof}
We will prove the well-definedness of the second definition.
Suppose, for $ a, c \in \mathcal{O}_d $ that
\[
\sigma = \mathrm{Mat}\left( \frac{a}{c} \right) = \begin{bmatrix} a & b \\ c & d \end{bmatrix}
\]
Then, by \cref{diamLm}, $ \sigma(V) $ will be the horoball based at $ \frac{a}{c} $ whose Euclidean diameter is $ \frac{1}{h |c|^2} $.
Since both the location and size of $ \sigma(V) $ are independent of the choice of $ b $ and $ d $,
the intersection $ \sigma(V) \cap \partial V $ will only depend on $ \frac{a}{c} $.
\end{proof}

We will use these circles to simplify the process of checking for covering.
To do this, we derive some properties about these circles.
First, it is immediately clear that the basepoint of the horoball
will have the same coordinate in the complex plane as the center of $ \mathrm{Circ}(\sigma) $.
Then, it remains to determine the radius of the circle.

\begin{lem}
Let $ A = \mathrm{Circ}\left( \begin{bmatrix} a & b \\ c & d \end{bmatrix} \right) $
where the main horoball, $ V $, has height $ h $.
Then, the radius of $ A $ will be
\[ \sqrt{\frac{1}{|c|^2} - h^2} \]
\end{lem}
\begin{proof}
Below is a diagram of the side view of the circle.
$ \mathrm{Circ}(\sigma) $ is the red line.

\bigskip
\begin{center} \begin{tikzpicture}[scale=3]
\draw[fill=blue!20] (1, 1) circle(1);
\node at (0, 1.75) {$ \sigma(\partial V) $};
\draw[thick] (-0.25, 8/13) -- (2.25, 8/13);
\node at (-1/3, 2/3) {$ \partial V $};
\draw[thick, red] (1/13, 8/13) -- (25/13, 8/13);

\draw (0, 0) -- (2, 0);
\draw[decorate, decoration={brace, raise=2pt}] (1, 0) -- (1, 2) node[pos=0.5, left=5pt]{$ \frac{1}{h |c|^2} $};
\draw[decorate, decoration={brace, mirror, raise=2pt}] (25/13, 0) -- (25/13, 8/13) node[pos=0.5, right=5pt]{$ h $};
\draw[decorate, decoration={brace, raise=2pt}] (1, 8/13) -- (25/13, 8/13) node[pos=0.5, above=5pt]{$ r $};
\draw[decorate, decoration={brace, raise=2pt}] (1, 0) -- (25/13, 8/13) node[pos=0.42, above=6pt]{$ s $};

\draw (25/13, 0) -- (25/13, 8/13) -- (1, 0) -- (1, 2) -- (25/13, 8/13);
\draw (0.92, 0) -- (0.92, 0.08) -- (1, 0.08);
\draw (25/13 - 0.08, 0) -- (25/13 - 0.08, 0.08) -- (25/13, 0.08);
\draw (1.15, 0) arc[start angle=0, delta angle=31.6, radius=0.15];
\draw (1, 2 - 0.15) arc[start angle=270, delta angle=31.6, radius=0.15];
\end{tikzpicture} \end{center}
\bigskip

Noticing the similar triangles, we have
\[
\frac{h}{s} = \frac{s}{\frac{1}{h |c|^2}}
\]
so $ s^2 = 1 / |c|^2 $.
Thus, $ r^2 = \frac{1}{|c|^2} - h^2 $.

\end{proof}
\medskip

\begin{lem} \label{tangLm}
For any two Euclidean spheres, $ A $ and $ B $, tangent to a plane $ P $,
$ A $ intersects $ B $ if and only if
\[ d^2 \leq 4sr \]
with tangency at equality
where $ r $ is the radius of $ A $, $ s $ is the radius of $ B $,
and $ d $ is the distance between $ A \cap P $ and $ B \cap P $.
\end{lem}
\begin{proof}
For a given $ r $ and $ s $, if two spheres are closer than
the distance necessary for tangency then they will have an intersection.
It remains to show that when $ d^2 = 4sr $, $ A $ and $ B $ are tangent.
If we consider a plane perpendicular to $ P $
passing through the centers of $ A $ and $ B $
then we get the following cross-sectional diagram.

\bigskip
\begin{center} \begin{tikzpicture}[scale=0.3]
\draw (9, 9) circle(9);
\node at (0, 15) {$ A $};
\draw (21, 4) circle(4);
\node at (25, 6) {$ B $};

\draw (0, 0) -- (25, 0);

\draw (9, 0) -- (9, 9);
\draw (21, 0) -- (21, 4);

\draw [thick, red] (9, 9) -- (21, 4) -- (9, 4) -- (9, 9);

\draw[decorate, decoration={brace, raise=2pt}] (9, 0) -- (9, 9) node[pos=0.5, left=6pt]{$ r $};
\draw[decorate, decoration={brace, raise=2pt, mirror}] (21, 0) -- (21, 4) node[pos=0.5, right=6pt]{$ s $};
\draw[decorate, decoration={brace, raise=2pt, mirror}] (9, 4) -- (21, 4) node[pos=0.5, below=6pt]{$ d $};

\end{tikzpicture} \end{center}

Then, clearly, $ (r + s)^2 = (r - s)^2 + d^2 $ so
\[ d^2 = (r^2 + 2rs + s^2) - (r^2 - 2rs + s^2) = 4rs \]
\end{proof}

\begin{lem}
Let $ A $ and $ B $ be closed Euclidean spheres tangent to the horizontal plane $ P $.
Suppose $ A $ has radius $ r $, $ B $ has radius $ s $, and $ d $ is the distance between $ P \cap A $ and $ P \cap B $.
If $ r \geq s $ then the height of the highest point in $ A \cap B $ is
\[ \begin{cases}
2s & \text{if } d^2 \leq 4s(r - s) \\
\frac{(s + r) + \sqrt{4sr - d^2}}{2 \left(1 + \left(\frac{s - r}{d}\right)^2\right)} & \text{otherwise} \\
\end{cases} \]
\end{lem}
\begin{proof}
If we consider a plane perpendicular to $ P $
passing through the centers of $ A $ and $ B $
then we get the following cross-sectional diagram.

\bigskip
\begin{center} \begin{tikzpicture}
\draw (3, 3) circle(3);
\node at (0, 5) {$ A $};
\draw (7, 2) circle(2);
\node at (9, 3) {$ B $};

\draw (0, 0) -- (9, 0);

\draw [thick, red] (3, 0) -- (5.92, 3.68) -- (7, 0);
\draw [thick, red] (5.92, 0) -- (5.92, 3.68);

\draw (3, 0) -- (3, 3) -- (5.92, 3.68) -- (7, 2) -- (7, 0);

\draw[thick] (3, 0.3) arc[start angle=90, delta angle=-38, radius=0.3];
\draw[thick] (5.92, 3.68 - 0.3) arc[start angle=270, delta angle=-38, radius=0.3];
\draw[thick] (7, 0.3) arc[start angle=90, delta angle=16, radius=0.3];
\draw[thick] (7, 0.35) arc[start angle=90, delta angle=16, radius=0.35];
\draw[thick] (5.92, 3.68 - 0.3) arc[start angle=270, delta angle=16, radius=0.3];
\draw[thick] (5.92, 3.68 - 0.35) arc[start angle=270, delta angle=16, radius=0.35];

\draw[decorate, decoration={brace, raise=2pt}] (3, 0) -- (3, 3) node[pos=0.5, left=6pt]{$ r $};
\draw[decorate, decoration={brace, raise=2pt}] (7, 2) -- (7, 0) node[pos=0.5, right=6pt]{$ s $};
\draw[decorate, decoration={brace, raise=2pt, mirror}] (3, 0) -- (5.92, 0) node[pos=0.5, below=6pt]{$ x $};
\draw[decorate, decoration={brace, raise=2pt, mirror}] (5.92, 0) -- (7, 0) node[pos=0.5, below=6pt]{$ y $};
\draw[decorate, decoration={brace, raise=2pt}] (5.92, 0) -- (5.92, 3.68) node[pos=0.5, left=6pt]{$ k $};

\end{tikzpicture} \end{center}
\bigskip

Then, from the similar triangles, we have the equations
\[ \begin{split}
d &= x + y \\
\frac{k}{\sqrt{k^2 + x^2}} &= \frac{\frac{\sqrt{k^2 + x^2}}{2}}{r} \qquad 2kr = k^2 + x^2\\
\frac{k}{\sqrt{k^2 + y^2}} &= \frac{\frac{\sqrt{k^2 + y^2}}{2}}{s} \qquad 2ks = k^2 + y^2 \\
\end{split} \]
which we can simplify to
\[ \begin{split}
2kr = k^2 + (d - y)^2 =& k^2 + y^2 + d^2 - 2yd = 2ks + d^2 - 2d \sqrt{2ks - k^2} \\
&\sqrt{2ks - k^2} = \frac{d}{2} + \frac{k}{d} (s - r) \\
2ks - k^2 &= \frac{d^2}{4} + \frac{k^2}{d^2} (s - r)^2 + k (s - r) \\
0 = k^2 &\left(1 + \left(\frac{s - r}{d}\right)^2 \right) - k (s + r) + \frac{d^2}{4} \\
k &= \frac{(s + r) \pm \sqrt{4sr - d^2}}{2\left(1 + \left(\frac{s - r}{d}\right)^2\right)}
\end{split} \]
Since we want the largest value of $ k $, we will take $ \pm $ to be $ + $.
Notice that if $ d < 4s(r - s) = r^2 - (2s - r)^2 $, though, then the apex of $ B $ (at height $ 2s $)
will lie inside $ A $ and so will always be the highest point.
\end{proof}

\medskip

\subsection{Finding the Generators}

\begin{defn}
Let $ V $ be a horoball based at $ \infty $ of height $ h > 0 $.
Let $ D $ be a fundamental domain of $ \mathbb{C} / \mathcal{O} $.
We define a set of ``generators at $ h $" to be a subset of $ \Gamma $,
\[
\mathrm{Gens}(h) = \left\{
\mathrm{Mat}\left( \frac{a}{c} \right)
\;\middle|\; a, c \in \mathcal{O}, |c| \leq \frac{1}{h},
\gcd(a, c) = 1, \frac{a}{c} \in D
\right\}.
\]
\end{defn}

As before, although we are using the expression $ \mathrm{Gens}(h) $ to represent this set,
the particular choice of representative for each horoball is arbitrary.

\begin{thm}
Let $ V $ be a horoball based at $ \infty $ of height $ h > 0 $.
Suppose $ S_\infty $ is a generating set for $ \Gamma_\infty $.
If
\[
D \subseteq \bigcup_{|c| \leq \frac{1}{h}} \left( \bigcup_{\gcd(a, c) = 1, \frac{a}{c} \in \overline{D}}
\mathrm{Circ}\left( \frac{a}{c} \right) \right)
\]
then,
\[
\mathrm{Gens}(h) \cup S_\infty
\]
will be a generating set for $ \Gamma $.
\end{thm}
\begin{proof}
We can prove this primarily by applying Macbeath's Theorem to $ V $ and $ \Gamma $.
First, to use Macbeath's Theorem, we must show that the orbit of $ V $ covers $ \textbf{H}^3 $.

If
\[
D \subseteq \bigcup_{|c| \leq \frac{1}{h}} \left( \bigcup_{\gcd(a, c) = 1, \frac{a}{c} \in \overline{D}}
\mathrm{Circ}\left( \frac{a}{c} \right) \right)
\]
then for all $ z \in D $, there exists $ a, c \in \mathcal{O} $ with $ z \in \mathrm{Circ}\left( \frac{a}{c} \right) $.
Thus, $ (z, h) \in \sigma(V) $ where $ \sigma = \mathrm{Mat}\left( \frac{a}{c} \right) \in \Gamma $.
Because this holds for all $ z \in D $, we can conclude, by \cref{ballcoverLm}, that $ \Gamma(V) $ covers $ \textbf{H}^3 $.

By Macbeath's Theorem, since the orbit of $ V $ covers $ \textbf{H}^3 $ and $ \textbf{H}^3 $ is simply connected,
the set $ \{\sigma \in \Gamma \;:\; \sigma(V) \cap V \neq \emptyset \} $ generates $ \Gamma $.
Now, we show that $ \mathrm{Gens}(h) $ and $ S_\infty $ will generate this set.
\medskip

Let $ \sigma \in \Gamma $ with $ \sigma(V) \cap V \neq \emptyset $ and
\[ \sigma = \begin{bmatrix} a & b \\ c & d \end{bmatrix} \]
We consider the cases when $ c = 0 $ and $ c \neq 0 $.

If $ c = 0 $ then $ 1 = ad - bc = ad $ which means $ a = d = \pm 1 $, because $ \mathcal{O} $ has no non-trivial units.
Further, because $ \Gamma = PSL_2(\mathcal{O}) $ only distinguishes transforms up to scaling, we can say without loss of generality that $ a = d = 1 $.
Thus,
\[ \sigma = \begin{bmatrix} 1 & b \\ 0 & 1 \end{bmatrix} \]
which fixes $ \infty $ so $ \sigma \in \Gamma_\infty $ and $ \sigma $ is generated by $ S_\infty $.

Otherwise, $ c \neq 0 $ and so by the definition of the fundamental domain
there exists $ a', s \in \mathcal{O} $ with $ \frac{a'}{c} \in D $ such that $ \frac{a'}{c} + s = \frac{a}{c} $.
Thus,
\[
\sigma =
\begin{bmatrix} 1 & s \\ 0 & 1 \end{bmatrix}
\cdot \begin{bmatrix} a' & b - sd \\ c & d \end{bmatrix}
= T_s \cdot \sigma'
\]
Of course, $ T_s $ is generated by $ S_\infty $.
Further, since $ \sigma(V) \cap V \neq \emptyset $,
the Euclidean diameter of the horoball $ \sigma(V) $ must exceed the height of $ V $,
\[ \frac{1}{h |c|^2} \geq h \]
so $ |c| \leq \frac{1}{h} $.
Also, we know $ \sigma \in \Gamma $ so $ a'd - c(b - sd) = 1 $ meaning $ \gcd(a', c) = 1 $.
From these two facts, we know there exist $ b', d' \in \mathcal{O} $ such that
\[
\tau = \begin{bmatrix}
a' & b' \\ 
c & d'
\end{bmatrix} \in \mathrm{Gens}(h)
\]
This leads to
\[
\sigma' = \begin{bmatrix} a' & b - sd \\ c & d \end{bmatrix}
= \begin{bmatrix} a' & b' \\ c & d' \end{bmatrix}
\cdot \begin{bmatrix} 1 & \frac{d - d'}{c} \\ 0 & 1 \end{bmatrix} = \tau \cdot T_{\frac{d - d'}{c}}
\]
and, obviously, $ T_{\frac{d - d'}{c}} $ is generated by $ S_\infty $ and $ \tau \in \mathrm{Gens}(h) $.
Finally, we conclude that every $ \sigma = T_s \cdot \tau \cdot T_{\frac{d - d'}{c}} $ is generated by $ \mathrm{Gens}(h) \cup S_\infty $.
\end{proof}

From the above theorem, we can find a generating set for $ \Gamma $
as long as we can find a height, $ h $, such that
\[
D \subseteq \bigcup_{|c| \leq \frac{1}{h}} \left( \bigcup_{\gcd(a, c) = 1, \frac{a}{c} \in \overline{D}}
\mathrm{Circ}\left( \frac{a}{c} \right) \right) .
\]
To check if the circles for a given height cover the fundamental domain,
we can use the method described in Appendix \ref{appendix:circlecover}.
\medskip

\begin{algorithm}
\caption{Find generators} \label{alg:find_gens}
\begin{algorithmic}
\Repeat
	\State Find denominators $ \left\{c \in \mathcal{O}_d \;:\; |c| < \frac{1}{h} \right\} $;
	
	\For{each denominator $ c $}
		\State Find all $ a \in \mathcal{O}_d $ coprime to $ c $;
		\State Calculate $ \sigma = \mathrm{Mat}\left( \frac{a}{c} \right) $;
		\State Add $ \sigma $ to $ \mathrm{Gens}(h) $;
		\State Add $ \mathrm{Circ}(\sigma) $ to $ \mathrm{Circles} $;
	\EndFor
\Until{$ D \subseteq \bigcup \mathrm{Circles} $}
\end{algorithmic}
\end{algorithm}

\subsection{Finding the Relations}

We now assume that we have an $ h $ such that $ \mathrm{Gens}(h) \cup S_\infty $ is a generating set for $ \Gamma $.
To obtain the relations,
we must find all the \emph{triple intersections}.
That is the intersections between $ V $, $ \sigma(V) $, and $ \tau(V) $ where $ \sigma, \tau \in \Gamma $.
This can be simplified substantially by removing certain redundancies.

\begin{lem} \label{decompLm}
For any $ \sigma \in \Gamma \setminus \Gamma_\infty $ such that $ \sigma(V) \cap V \neq \emptyset $,
\[ \sigma = T_r g T_s \]
for some $ g \in \mathrm{Gens}(h) $ and $ r, s \in \mathcal{O}_d $.
\end{lem}
\begin{proof}
Let $ \sigma = \begin{bmatrix} a & b \\ c & d \end{bmatrix} $ for $ a, b, c, d \in \mathcal{O}_d $.
First, by definition of the fundamental domain, there exists $ f \in D $ with $ \frac{a}{c} \in f + \mathcal{O}_d $.
So, $ a = cf + cr $ for $ r \in \mathcal{O}_d $ and take $ a' = cf $ so $ \frac{a'}{c} \in D $.
Then,
\[
\sigma = \begin{bmatrix} 1 & r \\ 0 & 1 \end{bmatrix}
\cdot \begin{bmatrix} a' & b - rd \\ c & d \end{bmatrix}
= T_r \cdot \sigma' .
\]

In order for $ \sigma(V) \cap V \neq \emptyset $, we know that the horoball $ \sigma(V) $ must be large enough.
Thus, its Euclidean diameter must be greater than $ h $ which by \cref{diamLm} means
\[ \frac{1}{h |c|^2} \geq h \]
Then, $ |c| \leq \frac{1}{h} $.
So by the definition of $ \mathrm{Gens}(h) $, there exists $ g = \mathrm{Mat}\left( \frac{a'}{c} \right) \in \mathrm{Gens}(h) $
with $ g = \begin{bmatrix} a' & b' \\ c & d' \end{bmatrix} $.
Then,
\[
g^{-1} \sigma' = \begin{bmatrix} d' & -b' \\ -c & a' \end{bmatrix} \cdot \begin{bmatrix} a' & b - rd \\ c & d \end{bmatrix}
= \begin{bmatrix} a'd' - b'c & (b - rd) d' - b'd \\ -a'c + a'c & - (b - rd) c + a'd \end{bmatrix}
= \begin{bmatrix} 1 & -rdd' \\ 0 & 1 \end{bmatrix} = T_s
\]
where $ s = -rdd' $.
Finally, $ \sigma = T_r \cdot \sigma' = T_r \cdot g \cdot T_s $.
\end{proof}

\begin{thm}
Assuming that the commutation relation between the generators of $ \Gamma_\infty $ is given,
all relations produced using Macbeath's Theorem are equivalent to a relation of one of the following forms
\[ A^{-1} T_a B = T_b C T_c \]
\[ A^{-1} = T_b C T_c \]
where $ A, B, C \in \mathrm{Gens}(h) $ and $ a, b, c \in \mathcal{O}_d $.
Further, these relations correspond to the triple intersections
$ A(V) \cap (T_a B)(V) \cap V $ and $ A(V) \cap V \cap V $, respectively.
\end{thm}
\begin{proof}
Suppose we have the triple intersection $ \sigma(V) \cap \tau(V) \cap V $ with $ \sigma, \tau \in \Gamma $.
First, we consider the trivial case $ \sigma(V) = \tau(V) = V $ so $ \sigma, \tau \in \Gamma_\infty $.
So let $ \sigma = T_1^a T_\omega^b  $ and $ \tau = T_1^c T_\omega^d $ for some $ a, b, c, d \in \mathbb{Z} $.
Then, clearly, the corresponding relation $ T_1^a T_\omega^b \cdot T_1^c T_\omega^d = T_1^{a + c} T_\omega^{b + d} $
is equivalent to the commutation relation $ T_1 \cdot T_\omega = T_\omega \cdot T_1 $.
\medskip

Before treating the other cases, we will take $ \phi = \sigma^{-1} \tau $ and notice
\[
\sigma^{-1}(\sigma(V) \cap \tau(V) \cap V) = V \cap \phi(V) \cap \sigma^{-1}(V) \neq \emptyset.
\]
We conclude $ \phi(V) $ intersects $ V $.

Second, we consider when $ \sigma(V) \neq V $ and $ \tau(V) = V $.
Thus, $ \tau \in \Gamma_\infty $ so $ \tau = T_t $.
Also, $ \sigma \notin \Gamma_\infty $ which means $ \phi = \sigma^{-1} \tau \notin \Gamma_\infty $.
Hence, because $ \sigma(V) \cap V \neq \emptyset $ and $ \phi(V) \cap V \neq \emptyset $,
using \cref{decompLm}, we know $ \sigma = T_r A T_s $ and $ \phi = T_u C T_v $
for $ A, C \in \mathrm{Gens}(h) $.
Then,
\[ T_u C T_v = \phi = \sigma^{-1} \tau = T_{-s} A^{-1} T_{-r} T_t \]
\[ A^{-1} = T_{u + s} C T_{v + r - t}. \]
Taking $ b = u + s $ and $ c = v + r - t $, we get the second form from the Lemma's statement.

Notice that
\[
T_{-r} (\sigma(V) \cap \tau(V) \cap V) T_{-s} = A(V) \cap T_{t - r - s}(V) \cap T_{- r - s}(V) = A(V) \cap V \cap V
\].
Then, the relation can be formed from the triple intersection $ A(V) \cap V \cap V \neq \emptyset $.
\medskip

Third, we consider when $ \sigma(V) \neq V \neq \tau(V) $ and $ \phi \in \Gamma_\infty $ so $ \phi = T_p $.
Using \cref{decompLm}, $ \sigma = T_r A T_s $ and $ \tau = T_x B T_y $.
Then,
\[ T_p = T_{-s} A^{-1} T_{-r} T_x B T_y \]
\[ A T_{p + s - y} = T_{x - r} B . \]
Now, take $ u = p + s - y $ and $ v = x - r $
as well as let $ A = \begin{bmatrix} a_1 & b_1 \\ c_1 & d_1 \end{bmatrix} $
and $ B = \begin{bmatrix} a_2 & b_2 \\ c_2 & d_2 \end{bmatrix} $.
Then,
\[
A T_u = \begin{bmatrix} a_1 & ua_1 + b_1 \\ c_1 & uc_1 + d_1 \end{bmatrix}
= \begin{bmatrix} a_2 + vc_2 & b_2 + vd_2 \\ c_2 & d_2 \end{bmatrix} = T_v B .
\]
so $ c_1 = c_2 $ and $ a_1 - a_2 = v c_2 $
which means $ \frac{a_1}{c_1} - \frac{a_2}{c_2} = \frac{a_1 - a_2}{c_2} \in \mathcal{O}_d $.
Since $ \frac{a_1}{c_1}, \frac{a_2}{c_2} \in D $, their difference must be zero and $ v = 0 $.
But if $ \frac{a_1}{c_1} = \frac{a_2}{c_2} $ then $ A = B $
since $ \mathrm{Gens}(h) $ has only a single matrix for each basepoint.
So the relation becomes $ A T_u = T_v B = B = A $ which is tautological.
\medskip

Fourth, we consider when $ \sigma(V) \neq V \neq \tau(V) $ and $ \phi \notin \Gamma_\infty $.
Using \cref{decompLm}, $ \sigma = T_r A T_s $, $ \tau = T_x B T_y $, and $ \phi = T_u C T_v $.
Then,
\[ T_u C T_v = \phi = \sigma^{-1} \tau = T_{-s} A^{-1} T_{-r} T_x B T_y \]
\[ A^{-1} T_{x - r} B = T_{u + s} C T_{v - y}. \]
Taking $ a = x - r $, $ b = u + s $ and $ c = v - y $, we get the first form from the Lemma's statement.

Notice that
\[ T_{-r} (\sigma(V) \cap \tau(V) \cap V) T_{-s} = A(V) \cap (T_{x - r} B T_{y - s})(V) \cap T_{-r-s}(V) = A(V) \cap (T_a B)(V) \cap V \]
Then, the relation $ A^{-1} T_a B = T_b C T_c $ can be generated
from the triple intersection $ A(V) \cap (T_a B)(V) \cap V \neq \emptyset $.
\end{proof}

\begin{thm}
For all $ X, Y \in \mathrm{Gens}(h) $
with $ X = \mathrm{Mat}\left(\frac{a_x}{c_x}\right) $ and $ Y = \mathrm{Mat}\left(\frac{a_y}{c_y}\right) $,
if $ X(V) \cap (T_s Y)(V) \neq \emptyset $ for some $ s \in \mathcal{O}_d $
then
\[ \mathrm{Tang}(c_x) := \frac{1}{|c_x| h} + \left|1 + \omega\right| \geq |s| \]
\end{thm}
\begin{proof}
Suppose $ X = \begin{bmatrix} a_x & b_x \\ c_x & d_x \end{bmatrix} $ and $ Y = \begin{bmatrix} a_y & b_y \\ c_y & d_y \end{bmatrix} $.
Then,
\[ T_s Y = \begin{bmatrix} a_y + sc_y & b_y + sd_y \\ c_y & d_y \end{bmatrix}. \]
Using \cref{diamLm}, we know $ X(V) $ and $ (T_s Y)(V) $
will have Euclidean diameters of $ \frac{1}{|c_x|^2 h} $ and $ \frac{1}{|c_y|^2 h} $.
Using \cref{tangLm}, since $ X(V) \cap (T_s Y)(V) \neq \emptyset $, we know
\[ \left| \frac{a_x}{c_x} - \frac{a_y + sc_y}{c_y} \right|^2 \leq 4 \frac{1}{2 |c_x|^2 h} \frac{1}{2 |c_y|^2 h} = \frac{1}{|c_x|^2 |c_y|^2 h^2} \]
\medskip

Now, because $ \frac{a_x}{c_x}, \frac{a_y}{c_y} \in D $,
we know $ \frac{a_x}{c_x} = \alpha + \beta \omega $ and $ \frac{a_y}{c_y} = \gamma + \delta \omega $.
where $ 0 \leq \alpha, \beta, \gamma, \delta < 1 $.
So taking $ \frac{a_x}{c_x} - \frac{a_y}{c_y} = (\alpha - \gamma) + (\beta - \delta)\omega = u + v\omega $.
We have $ 0 - 1 < u, v < 1 - 0 $ so $ u, v \in (-1, 1) $ meaning $ u^2, v^2 < 1 $ and $ uv \in (-1, 1) $.
Then,
\[
\left| \frac{a_y}{c_y} - \frac{a_x}{c_x} \right|^2 = (u + v\omega)(u + v\overline{\omega})
= u^2 + uv (\omega + \overline{\omega}) + v^2 |\omega|^2
\leq 1 + (\omega + \overline{\omega}) + |\omega|^2 = |1 + \omega|^2 .
\]
Then, by the triangle inequality,
\[
|s| \leq \left| \frac{a_x}{c_x} - \frac{a_y + sc_y}{c_y} \right| + \left| \frac{a_y}{c_y} - \frac{a_x}{c_x} \right| \leq \frac{1}{|c_x| |c_y| h} + |1 + \omega|
\]
Further, since $ c_y \neq 0 $ and $ c_y \in \mathcal{O}_d $, we know $ |c_y| \geq 1 $ so
\[ |s| \leq \frac{1}{|c_x| h} + \left|1 + \omega\right| \]
\end{proof}

From the above bound and the fact that $ \mathrm{Gens}(h) $ is finite,
it is clear that all possible $ A, B \in \mathrm{Gens}(h) $ and $ a \in \mathcal{O}_d $
can be searched for with
\[ A(V) \cap (T_a B)(V) \cap V \neq \emptyset \]
in finite time.

To actually determine the relations, we need to consider

\begin{algorithm}
\caption{Find relations} \label{alg:find_relats}
\begin{algorithmic}
\For{each denominator $ c $}
	\State $ S \gets \{ s \in \mathcal{O}_d \;:\; |s| \leq \mathrm{Tang}(c) \} $;
	\For{each $ A \in \mathrm{Gens}(h) $ where the basepoint of $ A $ has denominator $ c $}
		\State Find $ r, t \in \mathcal{O}_d $ and $ C \in \mathrm{Gens}(h) $ such that $ A^{-1} = T_r C T_t $;
		\State Add the relation $ A T_r C T_t = 1 $ to the list of relations;
		
		\For{each $ B \in \mathrm{Gens}(h) $}
			\For{each $ s \in S $}
				\If{$ A(V) \cap (T_a B)(V) \cap V \neq \emptyset $}
					\State $ C' \gets A^{-1} T_a B $;
					\State Find $ r, t \in \mathcal{O}_d $ and $ C \in \mathrm{Gens}(h) $ such that $ C' = T_r C T_t $;
					\State Add the relation $ B^{-1} T_{-s} A T_r C T_t = 1 $ to the list of relations;
				\EndIf
			\EndFor
		\EndFor
	\EndFor
\EndFor
\end{algorithmic}
\end{algorithm}

\newpage
\section{Results}

We have found presentations for the cases of $ d = -43, -67, -163 $
which along with the previously known cases of $ d = -1, -2, -3, -7, -11, -19 $
(discovered by Swan, see \cite{Sw})
completely describes the singly cusped Bianchi groups.
The presentations given below for $ d = -1, -2, -3, -7, -11, -19 $
come from Finis, Grunewald, and Tirao, see \cite{FGT}.

The abelianizations for the groups are
\[ \begin{split}
\ab{\Gamma_{-1}} &= C_2 \times C_2 \\
\ab{\Gamma_{-2}} &= C_6 \times C_\infty \\
\ab{\Gamma_{-3}} &= C_3 \\
\ab{\Gamma_{-7}} &= C_2 \times C_\infty \\
\ab{\Gamma_{-11}} &= C_3 \times C_\infty \\
\ab{\Gamma_{-19}} &= C_\infty \\
\ab{\Gamma_{-43}} &= C_\infty^2 \\
\ab{\Gamma_{-67}} &= C_\infty^3 \\
\ab{\Gamma_{-163}} &= C_\infty^7 \\
\end{split} \]
where $ C_n $ is the group with order $ n $ generated by a single element.

Clearly from the description of this method,
many redundant generators and relations are produced before simplifying the relation.
To perform this simplification, we used the Magma Computational Algebra System.
The precise commands involved in this can be found
in the files \texttt{simpl43.mgm}, \texttt{simpl67.mgm}, and \texttt{simpl163.mgm} on the GitHub page.
In the below table, the ``Raw" columns list the numbers of generators and relations
for the group presentations before simplification.
The ``Simplified" columns list the numbers for the presentations after being simplified by Magma.

\begin{center} \begin{tabular}{| c | c | c | c | c | c | c |}
\hline
\multirow{2}{*}{Ring} & \multicolumn{2}{c|}{Simplified} & \multicolumn{2}{c|}{Raw}
  & \multirow{2}{*}{Height} & \multirow{2}{*}{Depth} \\
\cline{2-5}
& Generators & Relations & Generators & Relations & & \\
\hline
$ \mathcal{O}_{-2} $ & 3 & 4 & 10 & 78 & 0.5000 & 4 \\
$ \mathcal{O}_{-7} $ & 3 & 4 & 10 & 52 & 0.5000 & 4 \\
$ \mathcal{O}_{-11} $ & 3 & 4 & 18 & 186 & 0.4220 & 5 \\
$ \mathcal{O}_{-19} $ & 4 & 7 & 34 & 407 & 0.3218 & 9 \\
$ \mathcal{O}_{-43} $ & 5 & 10 & 146 & 1986 & 0.2071 & 23 \\
$ \mathcal{O}_{-67} $ & 8 & 15 & 218 & 3311 & 0.1690 & 35 \\
$ \mathcal{O}_{-163} $ & 11 & 18 & 1290 & 25997 & 0.0982 & 103 \\
\hline
\end{tabular} \end{center}

Here, the ``Height" refers to the Euclidean height, $ h $, of the main horoball $ V $,
The ``Depth" is
\[ \max\left\{ |c|^2 \;:\; \begin{bmatrix} a & b \\ c & d \end{bmatrix} \in \mathrm{Gens}(h)\right\} \]
and so give some indication about the number of cases that needed to be searched to find the generators.

We can represent the generators in the following presentations with the matrices
\[
A = \begin{bmatrix} 1 & 1 \\ 0 & 1 \end{bmatrix} \qquad
B = \begin{bmatrix} 0 & 1 \\ -1 & 0 \end{bmatrix} \qquad
U = \begin{bmatrix} 1 & \omega \\ 0 & 1 \end{bmatrix}.
\]
For $ d = -1, -2, -3, -7, -11 $, we have
\[
\Gamma_{-1} = \langle
	A, B, U \mid
	B^2, (AB)^3, (BUBU^{-1})^3,
	AUA^{-1}U^{-1}, (BU^2BU^{-1})^2,
	(AUBAU^{-1}B)^2
\rangle
\]
\[
\Gamma_{-2} = \langle
	A, B, U \mid
	B^2, (AB)^3,
	AUAU^{-1}, (BU^{-1}BU)^2
\rangle
\]
\[ \begin{split}
\Gamma_{-3} &= \langle
	A, B, U \mid
	B^2, (AB)^3, AUA^{-1}U^{-1}, \\
	&(UBA^2U^{-2}B)^2, (UBAU^{-1}B)^3,
	AUBAU^{-1}BA^{-1}UBA^{-1}UBAU^{-1}B
\rangle
\end{split} \]
\[
\Gamma_{-7} = \langle
	A, B, U \mid
	B^2, (BA)^3, AUA^{-1}U^{-1},
	(BAU^{-1}BU)^2
\rangle
\]
\[
\Gamma_{-11} = \langle
	A, B, U \mid
	B^2, (BA)^3, AUA^{-1}U^{-1},
	(BAU^{-1}BU)^3
\rangle.
\]
For $ d = -19 $, we can represent the $ C $ generator by the matrix
\[ C = \begin{bmatrix} 1 - \omega & 2 \\ 2 & \omega \end{bmatrix} \].
Then, the presentation is
\[
\Gamma_{-19} = \langle
	A, B, C, U \mid
	B^2, (BA)^3, AUA^{-1}U^{-1},
	C^3, (CA^{-1})^3, (BC)^2,
	(BA^{-1} UCU^{-1})^2
\rangle.
\]
For $ \mathcal{O}_{-43} $, we can represent $ R $ and $ S $ with the matrices
\[
R = \begin{bmatrix} 1 + \omega & \omega - 6 \\ 2 & \omega \end{bmatrix}
\qquad
S = \begin{bmatrix} 3 & \omega \\ 1 - \omega & 4 \end{bmatrix}.
\]
Then, the presentation is
\[ \begin{split}
\Gamma_{-43} &= \langle
	A, B, U, R, S \mid \\
	&B^2, (BA)^3, AUA^{-1}U^{-1}, (R^{-1} U)^3, \\
	&A^{-1} R A^{-1} U^{-1} R A^{-1} R^{-1} U R^{-1}, \\
	&A^{-1} R U^{-1} S R^{-1} U B S^{-1} B, \\
	&(U A S^{-1} B S R^{-1})^2, \\
	&(B S U^{-1} R S^{-1} A)^2, \\
	&A^{-1} U^{-1} S R^{-1} U S^{-1} B A^{-1} S R^{-1} U S^{-1} B U, \\
	&U B U^{-1} A^{-1} R S^{-1} B S A^{-1} B S^{-1} B U R^{-1} A S R^{-1}
\rangle.
\end{split} \]

For $ \mathcal{O}_{-67} $, the generators additional generators can be represented by
\[
K = \begin{bmatrix} 2\omega & -23 + 2\omega \\ 3 & 1 + 2\omega \end{bmatrix} \qquad
L = \begin{bmatrix} 2 + \omega & -4 + \omega \\ 4 & 1 + \omega \end{bmatrix} \qquad
M = \begin{bmatrix} 1 + 3\omega & -13 + \omega \\ 4 & \omega \end{bmatrix}
\]
\[
N = \begin{bmatrix} -2 + \omega & -10 \\ \omega & -9 + \omega \end{bmatrix} \qquad
P = \begin{bmatrix} 2 & \omega \\ 1 - \omega & 9 \end{bmatrix}.
\]
The presentation is
\[ \begin{split}
\Gamma_{-67} &= \langle
	A, B, U, K, L, M, N, P \mid \\
	&B^2, A U A^{-1} U^{-1}, \\
	&U P^{-1} A^{-1} N P^{-1} B, \\
	&(B U^{-1} M)^2, (A B)^3, (M U^{-1})^3, \\
	&B U P^{-1} B U N^{-1} A, \\
	&P^{-1} A^{-1} N A N^{-1} B A^{-1} P A, \\
	&L^{-1} B A^{-1} L U^{-1} K P^{-1} B U K^{-1} U, \\
	&L B U^{-1} M L^{-1} B A^{-1} N K^{-1} U M^{-1} K A^{-1} U^{-1}, \\
	&A U L^{-1} A B L M^{-1} U A L^{-1} A B L M^{-1}, \\
	&L^{-1} A U K^{-1} U P^{-1} B A^{-1} K A^{-1} U^{-1} L B A, \\
	&U^{-1} L B U^{-1} M L^{-1} A U K^{-1} M U^{-1} K P^{-1} A^{-1} N, \\
	&(M L^{-1} B A^{-1} L P^{-1} B A^{-1})^3, \\
	&(L^{-1} A B L M^{-1} A B A^{-1} N)^3
\rangle.
\end{split} \]

\newpage
For $ \mathcal{O}_{-163} $, we can represent the additional generators with
\[
D = \begin{bmatrix} 1 + \omega & -21 + \omega \\ 2 & \omega \end{bmatrix} \qquad
E = \begin{bmatrix} 1 + \omega & -10 + \omega \\ 4 & 2 + \omega \end{bmatrix} \qquad
F = \begin{bmatrix} 3 + \omega & -23 + 3\omega \\ 5 & 3 + 3\omega \end{bmatrix} \qquad
G = \begin{bmatrix} 3 + \omega & -6 + \omega \\ 6 & 2 + \omega \end{bmatrix}
\]
\[
H = \begin{bmatrix} 15 & 4\omega \\ 1 - \omega & 11 \end{bmatrix} \qquad
T = \begin{bmatrix} 31 & 3\omega \\ 1 - \omega & 4 \end{bmatrix} \qquad
V = \begin{bmatrix} -7 + \omega & -12 \\ \omega & -6 + \omega \end{bmatrix} \qquad
W = \begin{bmatrix} 2 + 4\omega & -23 + 2\omega \\ 7 & 2 + \omega \end{bmatrix}.
\]
The presentation is
\[ \begin{split}
\Gamma_{-163} &= \langle A, U, B, D, E, F, G, H, T, V, W \mid \\
	&B^2, A U A^{-1} U^{-1}, (U D^{-1})^3, (A B)^3, \\
	&(T B U^{-1})^3, A^{-1} D^{-1} U D^{-1} A^{-1} D U^{-1} A^{-1} D, \\
	&D^{-1} W G^{-1} A^{-1} V A^{-1} B W^{-1} D U^{-1}
	W A V^{-1} G A^{-1} B W^{-1} U,
	(A E U^{-1} T B A E^{-1})^3, \\
	&B T^{-1} U H^{-1} U A F^{-1} V A^{-1} B A
	V^{-1} G B A G^{-1} V A^{-1} B A V^{-1} F
	A^{-1} U^{-1} H, \\
	&(V A^{-1} B U^{-1} T W^{-1} U D^{-1} W B A
	V^{-1} A B A^{-1})^2, \\
	&(E A^{-1} B T^{-1} U E^{-1} A^{-1} V A^{-1} B
	W^{-1} D U^{-1} W B A V^{-1})^2, \\
	&(H U^{-1} T B A E^{-1} A^{-1} V A^{-1} B W^{-1}
	D U^{-1} W B A V^{-1} F A^{-1} D^{-1} A U \\ 
	&\quad F^{-1} E A^{-1} U^{-1} T B H^{-1} U A F^{-1}
	V A^{-1} B A V^{-1} F U^{-1}), \\
	&(G E^{-1} H U^{-1} T B D^{-1} W G^{-1} E A^{-1}
	B T^{-1} U H^{-1} A^{-1} D U^{-1} H U^{-1} T B \\
	&\quad A E^{-1} G W^{-1} D B T^{-1} U H^{-1} E
	A^{-1} B W^{-1} U D^{-1} W G^{-1} B A^{-1}), \\
	&(A^{-1} E A^{-1} B T^{-1} U E^{-1} H U^{-1} T B
	U^{-1} W B A G^{-1} V A^{-1} B A V^{-1} F \\
	&\quad A^{-1} D^{-1} U A F^{-1} G W^{-1} D B T^{-1} U
	H^{-1} U A F^{-1} V A^{-1} B A V^{-1} F U^{-1}), \\
	&(A F^{-1} E A^{-1} B T^{-1} U E^{-1} A^{-1} V
	A^{-1} B W^{-1} D U^{-1} W B A V^{-1} F \\
	&\quad A^{-1} U^{-1} D A F^{-1} V A^{-1} B W^{-1} U
	D^{-1} W B A V^{-1} A E U^{-1} T B A E^{-1} F U^{-1} A^{-1} D), \\
	&(B A V^{-1} A B A^{-1} V A^{-1} B G^{-1} B
	A^{-1} G E^{-1} H U^{-1} T B D^{-1} W G^{-1} E \\
	&\quad A^{-1} B T^{-1} U H^{-1} A^{-1} D A F^{-1} V
	A^{-1} B A V^{-1} G A^{-1} B W^{-1} U A D^{-1} \\
	&\quad W G^{-1} F U^{-1} H U^{-1} T B A E^{-1} G
	W^{-1} D B T^{-1} U H^{-1} E), \\
	&(B T^{-1} U H^{-1} U F^{-1} V A^{-1} B A V^{-1}
	F A^{-1} U^{-1} H U^{-1} T B D^{-1} W G^{-1} \\
	&\quad F U^{-1} H U^{-1} T B A E^{-1} G W^{-1} D
	B T^{-1} U H^{-1} E B A V^{-1} A B A^{-1} \\
	&\quad V A^{-1} B E^{-1} H U^{-1} T B D^{-1} W G^{-1}
	E A^{-1} B T^{-1} U H^{-1} U F^{-1} G W^{-1} D A^{-1}), \\
	&(W G^{-1} F U^{-1} H U^{-1} T B A E^{-1} G
	W^{-1} D B T^{-1} U H^{-1} E U^{-1} T B A \\
	&\quad E^{-1} H U^{-1} T B D^{-1} W G^{-1} E A^{-1} B
	T^{-1} U H^{-1} A U F^{-1} V A^{-1} B A V^{-1} G A^{-1} \\
	&\quad B W^{-1} D U^{-1} A^{-1} W G^{-1} F U^{-1} H U^{-1} T B
	A E^{-1} G W^{-1} D B T^{-1} U H^{-1} E U^{-1} T B A E^{-1} \\
	&\quad H U^{-1} T B D^{-1} W G^{-1} E A^{-1} H^{-1} U A
	F^{-1} V A^{-1} B A V^{-1} G B A W^{-1} D U^{-1} A^{-1})
\rangle.
\end{split} \]

\newpage \appendix
\section{Circle Covering Algorithm} \label{appendix:circlecover}

In order to check if a set of horoballs (and their shifts) cover the boundary of $ V $,
we can instead check that the circular intersections of the horoballs cover the fundamental domain.
One could check this manually by drawing all of the circles,
but for hundreds of generators this becomes prohibitive.
Instead, we describe an algorithm that can verify if a set of circles covers a parallelogram.

To do this, we subdivide the fundamental domain into an $ n \times n $ grid of parallelograms $ \frac{1}{n^2} $ the size of the original.
We then check that each sub-parallelogram is entirely covered by one of the circles in the set.
If this is the case for all sub-parallelograms then certainly the fundamental domain is covered.
It should be noted, though, that this method can produce false negatives if the circles do not have sufficient overlap.

\begin{algorithm}
\caption{Check if a set of Generators is complete} \label{alg:check_gens}
\begin{algorithmic}
\State \[
\text{Subdivide $ D $ into }
P \gets \left\{\left[\frac{i}{n}, \frac{i + 1}{n}\right] + \left[\frac{j}{n}, \frac{j + 1}{n}\right]\omega \;:\; i, j \in \{0, 1, \ldots, n\}\right\}
\]

\For{each parallelogram $ A \in P $}
	\For{$ \sigma \in \mathrm{Gens}(h) $}
		\If{$ A \subseteq \mathrm{Circ}(\sigma) $}
			\State Go to next parallelogram;
		\EndIf
	\EndFor
	\If{No circles contain $ A $}
		\State $ \mathrm{Gens}(h) $ fails to be complete;
	\EndIf
\EndFor

\State $ \mathrm{Gens}(h) $ along with the generators of $ \Gamma_\infty $ are complete;
\end{algorithmic}
\end{algorithm}

\begin{center}
\includegraphics[width=0.85\textwidth]{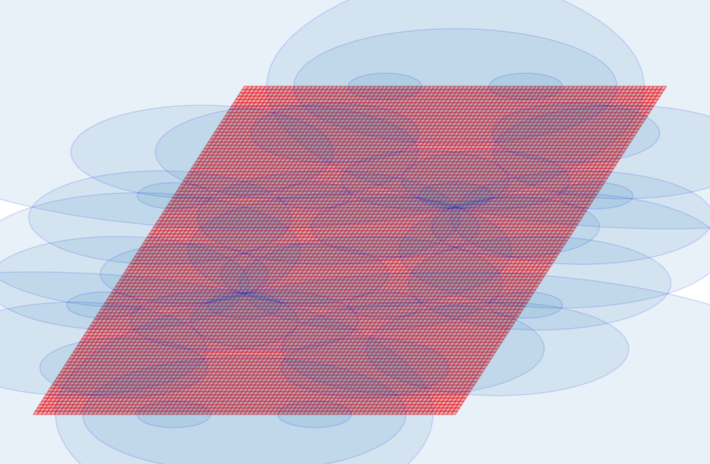}
\captionof{figure}{
Circles for each generator of $ \Gamma_{-19} $ covering the fundamental domain (red parallelogram).
The domain is split into a $ 100 \times 100 $ grid to apply the above algorithm.
}
\end{center}

\section{Finding Representatives of Residue Classes} \label{appendix:resrepr}

In the process of finding matrices with a given $ c $ value,
it is necessary to search through all residue classes of $ \mathcal{O}_d / c\mathcal{O}_d $.
To do this, we need representatives of each of these classes.

\begin{lem}
If $ c \in \mathcal{O}_d $ and $ D $ is a fundamental domain of $ \mathbb{C} / \mathcal{O}_d $ then
\[ R = \left\{ a \;:\; \frac{a}{c} \in D \right\} \]
is a complete set of representatives for $ c\mathcal{O}_d $.
\end{lem}
\begin{proof}
Suppose $ x + c \mathcal{O}_d $ is a residue class.
Then, by definition of the fundamental domain,
there exists $ s \in D $ such that $ s - \frac{x}{c} = k \in \mathcal{O}_d $.
Then, take $ a = sc = ck + x \in x + c\mathcal{O}_d $ so that $ \frac{a}{c} \in D $.
Thus, $ a \in R $ and $ a \sim x $ meaning every residue class is represented in $ R $.
\end{proof}

It is also necessary during the computations
to find the representative of the residue class that an arbitrary element belongs to.

\begin{lem}
For any $ a, c \in \mathcal{O}_d $,
there exist $ q, r \in \mathcal{O}_d $
such that $ a = qc + r $ and $ \frac{r}{c} \in D $.
\end{lem}
\begin{proof}
Suppose $ a, c \in \mathcal{O}_d $ with $ c = x_c + y_c \omega $ and $ a = x_a + y_a \omega $.
Further, we assume $ \omega + \overline{\omega} = t $ and $ \omega \overline{\omega} = n $.
Then, take
\[
q = \left\lfloor \frac{x_a x_c + t x_a y_c + n y_a y_c}{|c|^2} \right\rfloor +
\left\lfloor \frac{x_c y_a - x_a y_c}{|c|^2} \right\rfloor \omega = \lfloor x' \rfloor + \lfloor y' \rfloor \omega
\]
and simply take $ r = a - qc $.
It remains to show that $ \frac{r}{c} \in D $.

First,
\[ \begin{split}
\frac{r}{c} = \frac{a}{c} - q &= \frac{ac}{c \overline{c}} - q \\
&= \frac{(x_a + y_a \omega)(x_c + y_c \overline{\omega})}{|c|^2} - q \\
&= \frac{x_a x_c + y_a x_c \omega + x_a y_c \overline{\omega} + y_a y_c \omega \overline{\omega}}{|c|^2} - q \\
&= \frac{x_a x_c + y_a x_c \omega + t x_a y_c - x_a y_c \omega + n y_a y_c}{|c|^2} - q \\
&= x' + y'\omega - \lfloor x' \rfloor - \lfloor y' \rfloor \omega \\
&= (x' - \lfloor x' \rfloor) + (y' - \lfloor y \rfloor) \omega
\end{split} \]
and by definition of $ \lfloor x' \rfloor $, we know $ 0 \leq x' - \lfloor x' \rfloor < 1 $.
Similarly, $ 0 \leq y' - \lfloor y' \rfloor < 1 $ so $ \frac{r}{c} \in D $.
\end{proof}


\medskip
\begin{flushleft}
\textsc{Tanner Reese \\
School of Mathematical and Statistical Sciences, Arizona State University}

\texttt{treese5@asu.edu}
\end{flushleft}

\end{document}